\documentclass[11pt, one side,emlines]{amsart}
\usepackage[a4paper,margin=1in]{geometry}
\usepackage{amssymb,latexsym,xy,eucal,mathrsfs}
\textwidth=18cm \textheight=27cm \theoremstyle{plain}
\newtheorem{theorem}{Theorem}[section]
\newtheorem{lemma}[theorem]{Lemma}

\newtheorem{corollary}[theorem]{Corollary}
\newtheorem{definition}[theorem]{Definition}

\numberwithin{equation}{section} 
\large

\begin{document}

\title{  On $k$-Connected $\Gamma$-Extensions of Binary Matroids}
\author{Y. M. Borse$^1$ and Ganesh Mundhe$^2$ }
\email{\emph{$^1$ymborse11@gmail.com; $^2$ganumundhe@gmail.com}}
\address{\rm $^1$ Department of Mathematics, Savitribai Phule Pune University, Pune-411007, India.}
\address{\rm $^2$ Army Institute of Technology, Pune-411015, India.}\thanks{ This research is supported by DST-SERB, Government of India under the project SR/S4/MS:750/12}
 \maketitle
 \begin{abstract} Slater introduced the  point-addition operation on graphs to classify 4-connected graphs. The $\Gamma$-extension operation on binary matroids is a  generalization of  the point-addition operation. In this paper, we obtain necessary and sufficient conditions to preserve $k$-connectedness of a binary matroid under the $\Gamma$-extension operation. We also obtain a necessary and sufficient condition to get a connected  matroid  from a disconnected binary matroid using the  $\Gamma$-extension operation.
 \end{abstract}
   \section{Introduction}
\noindent
 We refer to \cite {ox} for standard terminology in graphs and matroids. The matroids considered here are loopless and coloopless. Slater \cite{sl} introduced the point-addition operation on graphs and used it to classify $4$-connected graphs. Azanchiler \cite{a1} extended this operation to  binary matroids as follows: 
 \begin{definition} \cite{a1}  Let $M$ be a binary matroid with ground set $S$ and standard matrix representation $A$ over $GF(2).$ Let $X = \{x_1, x_2, \dots, x_m\} \subset S$ be an independent set in $M$ and let $\Gamma = \{\gamma_{1}, \gamma_{2}, \dots, \gamma_{m}\}$ be a set such that $ S \cap \Gamma = \phi.$    Suppose  $A'$ is  the matrix obtained from the matrix $A$ by adjoining $m$ columns  labeled by $\gamma_{1}, \gamma_{2},...,\gamma_{m}$ such that the column labeled by $\gamma_{i}$
is same as the column labeled by $x_i$ for $ i =1, 2, \dots, m.$ Let $A^X$ be the matrix obtained by adjoining one extra row to  $A'$ which has entry 1 in the column labeled by $\gamma_{i}$ for $i= 1, 2 \dots, m$  and zero elsewhere. The  vector matroid of the matrix $A^{X},$ denoted by $M^{X},$  is called as the $\Gamma$-extension of $M$ and the transition from $M$ to $M^{X}$ is called as $\Gamma$-extension operation on $M.$ 
\end{definition}

An example given at the end of the paper illustrates the definition.
Note that the ground set of the matroid $M^{X}$ is $S\cup \Gamma$  and $M^X\setminus \Gamma  = M.$ Therefore $M^X$ is an extension of $M.$   The $\Gamma$-extension operation is related to the {\it splitting operation}  on binary matroids, which is defined by Shikare et al. \cite{saw}, as follows: 
\begin{definition} \cite{saw}  Let $M$ be a binary matroid with standard matrix representation $A$ over  $GF(2)$ and let $Y$ be a non-empty set of elements of $M.$  Let $A_Y$ be the matrix  obtained by adjoining one extra row to the matrix $A$ whose entries are 1 in the columns labeled by the elements of the set $Y$ and zero otherwise.  The vector matroid of the matrix $A_Y,$ denoted by $M_Y,$  is called as the splitting matroid of  $M$ with respect to $Y,$ and the  transition from $M$ to $M_{Y}$ is called as the {\it splitting operation} with respect to $Y.$  
 \end{definition} 
 Let $M$ be a binary matroid with ground set $S$ and let $X = \{x_1, x_2, \dots, x_m\}$  be an independent set in $M.$ Obtain the extension $M'$ of $M$ with ground set $S \cup \Gamma,$ where $\Gamma = \{\gamma_1, \gamma_2, \dots, \gamma_m\}$ is disjoint from $S,$ such that  $\{x_i, \gamma_i\}$ is a 2-circuit in $M'$ for each $i.$ The matroid $M'_{\Gamma}$ obtained from $M'$ by splitting the set $\Gamma$  is the $\Gamma$-extension matroid $M^X.$ 

The splitting operation with respect to a pair of elements, which is a special case of Definition 1.2, was earlier defined by Raghunathan et al. \cite{rsw}  for binary matroids as an extension of the corresponding graph operation due to Fleischner \cite{f}. 

Whenever we write $M^{X},$ it is assumed that $X$ is a non-empty independent set of  the matroid $M.$ 

Azanchiler \cite{a1} characterized the circuits and the bases of the  $\Gamma$-extension  matroid $M^{X}$ in terms the circuits and bases of $M,$ respectively. Some results on preserving graphicness of  $M$ under the $\Gamma$-extension operation are obtained in \cite{a2}. Borse and Mundhe \cite{bm2} characterized the binary matroids $M$ for which $M^X$ is graphic for any independent set $X$ of $M.$

A {\it $k$-separation} of a matroid $M$  is a partition of its ground set $S$ into two disjoint sets $A$ and $B$ such that  $min~\{\left| A \right|,\left| B \right| \} \geq k$ and $ r(A)+r(B)-r(M)\leq k-1.$  A matroid $M$ is {\it $k$-connected} if it does not have a $(k-1)$-separation. Also,  $M$ is  {\it connected} if it is  2-connected.

In general, the splitting operation does not preserve the connectivity of a given matroid. Borse and Dhotre \cite{bd} provided a sufficient condition to preserve connectedness of a matroid   while Borse \cite{b} gave a sufficient condition to get a  $k$-connected matroid from given the $(k+1)$-connected binary matroid, under the splitting with respect to a pair of elements. Borse and Mundhe \cite{bm1}, and Malwadkar et al. \cite{msd} gave two characterizations for getting a $k$-connected  matroid from the given $(k+1)$-connected binary matroid by splitting with respect to any set of $k$ elements.


 The $\Gamma$-extension operation also does not give $k$-connected matroid from the given $k$-connected binary matroid in general. Azanchiler \cite{a1} obtained sufficient conditions to preserve $2$-connectedness and 3-connectedness  of a binary matroid under this operation. 
 
In this paper, we obtain  necessary and sufficient conditions to preserve $k$-connectedness  under  the   $\Gamma$-extension operation for any integer $ k \geq 2.$  We also give necessary and sufficient conditions to get a \textit{connected}  matroid  from a disconnected binary matroid  in terms of the  $\Gamma$-extension operation.

 \section{Proofs}
 \noindent
We  need some lemmas.
\begin{lemma}\cite{a1}\label{1}  Let $M$ be a binary matroid with ground set $S$ and let $X$ be an independent set in $M.$ Suppose $M^X$ is  the  $\Gamma$-extension of $M$ with ground set $ S \cup \Gamma.$  Let $r$ and $r'$ be the rank functions of $M$ and $M^{X},$ respectively.  Then\\
   (i) $\Gamma$ is independent in $M^{X};$\\
	(ii) $r'(A) = r(A) $ if $ A \subset S;$ \\
	(iii) $r'(A) \geq r(S\cap A) + 1 $ if  $ A$ intersects $\Gamma;$  \\
	(iv) $r'(M^{X})=r(M)+1.$\\
\end{lemma}

\begin{lemma} \cite{a1} \label{2}
Let $M$ be a binary matroid with ground set $S$ and let $X$ be an independent set in $M.$ Then $ Z \subset S\cup \Gamma$ is a circuit of $M^X$ if and only if one of the following conditions holds:\\
(i) $Z$ is a circuit of $M;$ \\
	(ii) $  Z  = \{x_i,  x_j, \gamma_{i}, \gamma_{j}\}$ for some distinct elements $x_i, x_j$ of $X$ and the corresponding elements $\gamma_i, \gamma_j$ of $\Gamma;$\\
(iii) $ Z = J \cup (D - X_J),$ where $ J \subset \Gamma$ with $ |J|$ even and $D$ is a circuit of $M$ containing the set $ X_J = \{ x_i \in X\colon \gamma_i \in J\}.$ 
	\end{lemma}
	\begin{lemma}  [\cite{ox}, pp 273] \label{3} Let $M$ be a $k$-connected  matroid with at least $2(k-1)$ elements. Then every circuit and every cocircuit of $M$ contains at least $k$ elements.
\end{lemma} 
The next lemma is a consequence of [9, Proposition 2.1.6].
\begin{lemma}\cite{b} \label{4} Let $M$ be a matroid with ground set $S$ and let $ Y \subset S $ such that $ r(M\setminus Y) = r(M) - 1.$ Then $Y$ contains a cocircuit of $M.$ 
\end{lemma}
The following result follows immediately from Lemma \ref{3} and  Lemma \ref{4}. 
	\begin{corollary}  \label{5}
	Let $M$ be a $k$-connected matroid with ground set $S$ such that $|S| \geq 2(k-1).$ Then  $r(M \backslash Y) = r(M)$ for any $Y \subset S$ with $|Y|< k.$ 
	\end{corollary}

We now give necessary and sufficient conditions to obtain a $k$-connected matroid from the given $k$-connected binary matroid as follows. 
\begin{theorem} \label{6}
Let $k \geq 2$ be an integer and $M$ be a $k$-connected  binary matroid with at least $2(k-1)$ elements and $X$ be  an independent set in $M.$ Then the $\Gamma$-extension matroid $M^{X}$  is $k$-connected if and only if $|X|\geq k$ and $2 \leq k \leq 4$.
\end{theorem}
\begin{proof}
Suppose  $|X|\geq k$ and $2\leq k \leq 4.$ We prove that $M^X$ is $k$-connected. The ground set of $M^X$ is $ S \cup \Gamma,$ where $\Gamma$ is disjoint from the ground set $S $ of $M.$  Since $ |\Gamma| = |X|,$ $|\Gamma|\geq k.$ By Lemma \ref{1}(i), $\Gamma$ is independent in $M^X.$  Suppose $r$ and $r'$ denote the rank functions of $M$ and $M^X,$ respectively. Assume that $M^X$ is not $k$-connected.  Then $M^X$  has a $(k-1)$-separation $(A,B).$ Therefore $ A$ and $ B$ are  non-empty disjoint subsets of $S\cup \Gamma$ such that $S\cup \Gamma = A \cup B$ and further,   
\begin{center} min $\{\left| A \right|,\left| B \right| \} \geq k-1~~~~{\rm and}$\\
$ r'(A)+r'(B)-r'(M^{X})\leq k-2.~~~~~\ldots\ldots (1)$
\end{center}
As  $A$ and $B$ are non-empty, each of them intersects $S$ or $ \Gamma$ or both. We consider the three cases depending on whether $A$ intersect only $S$ or only $\Gamma$ or both and obtain a contradiction in each of these cases.
\vskip.15cm\noindent
{\bf Case (i).} $A$ intersects only $\Gamma.$  
\vskip.15cm\noindent
As $A\subset \Gamma,$  $B= S\cup  (\Gamma - A).$ Since  $ \Gamma$ is independent,  $A$ is independent in $M^X.$ Consequently,  $ r'(A) = |A|\geq k-1.$  Suppose  $ A \neq \Gamma.$ Then, by Lemma \ref{1}(iii) and (iv), $r'(B) \geq r(S) + 1= r(M) + 1 = r'(M^X).$ Therefore $r'(B)= r'(M^X).$  Hence  $r'(A)+r'(B)-r'(M^{X})\geq k-1,$ which contradicts (1). Therefore   $A = \Gamma.$ Hence $B = S$ and $ r'(A) = |\Gamma| \geq k.$ By Lemma \ref{1}(ii) and (iv), $r'(B) = r'(S) = r(S) = r(M) = r'(M^X) - 1.$ Therefore $r'(A)+r'(B)-r'(M^{X})\geq k-1,$ which is a contradiction to (1). 
\vskip.15cm\noindent
{\bf Case (ii).} $A$ intersects only $S.$
\vskip.15cm\noindent
As $ A\cap \Gamma = \phi ,$ $A\subset S$ and  $B= (S - A) \cup \Gamma.  $  Therefore, by Lemma \ref{1}(i) and (ii), $ r'(A) = r(A)$ and $r'(B) \geq r'(\Gamma)  = |\Gamma| \geq k.$  Suppose $ |S- A| \leq k-2.$ Then, by Corollary \ref{5}, $r(A)=r(M) $. Consequently, by Lemma \ref{1}(iv),  
\begin{center}
$ r'(A)+r'(B)-r'(M^{X}) = r(A) + r'(B)  - (r(M) +  1) \geq r'(B) - 1 \geq k-1,$
\end{center}
which is a contradiction to (1).
Hence $|S - A| \geq k-1. $ By Lemma \ref{1} (ii) and (iii), $r(S - A) = r'(S - A)\leq r'(B)-1.$ Therefore, by Inequality (1),
\begin{center}
$ r(A)+r(S - A)-r(M)\leq r'(A)+r'(B)-1-r(M^{X})+1 \leq k-2.$
\end{center}
This shows that $A$ and $S-A$ gives a $(k-1)$-separation of $M,$ which is a contradiction to fact that $M$ is $k$-connected.
 \vskip.15cm\noindent
{\bf Case (iii).} $A$ intersects both $S$ and $ \Gamma.$ 
\vskip.15cm\noindent
 Let $ S_1 = A\cap S$ and $\Gamma_{1} = A \cap \Gamma.$ Since $B \neq \phi,$ it  intersects $S$ or $\Gamma.$    If  $ B $ intersects only $S$ or only $\Gamma,$ then we get a contradiction by interchanging roles of $A$ and $B$ in  Case (i) and Case (ii). Therefore  $B$ intersects both $S$ and $\Gamma.$  Let $ S_2 = B\cap S $ and   $ \Gamma_{2}= B\cap \Gamma.$  Then $ S_i \neq \phi$ and $ \Gamma_i \neq \phi$ for $ i =1, 2.$   By  Lemma \ref{1}(ii) and (iii), $r(S_1) = r'(S_1)  \leq r'(A) - 1 $ and $r(S_2) = r'(S_2)   \leq r'(B) - 1 .$ By (1), 
\begin{center}
$r(S_{1})+r(S_{2})-r(M) \leq r'(A)-1+r'(B)-1-r'(M^{X})+1\leq k-3.$
\end{center}
Hence, if $|S_{1}|\geq k-2$ and $|S_{2}|\geq k-2,$  then $(S_{1},S_{2})$ gives a  $(k-2)$-separation of $M,$ a contradiction to fact that $M$ is $k$-connected. Consequently, $|S_{1}|\leq k-3$ or $|S_{2}|\leq k-3.$  


Suppose  $|S_1| \leq  k-3.$  As $ k \leq 4$ and $1\leq |S_1|,$ $k=4$ and $|S_1 | = k-3 = 1.$ Thus $A$ contains exactly one element, say $x,$  of $M.$  Further, $|A| \geq k -1 = 4-1 =3.$  We claim that $r'(A)\geq 3.$ Suppose $r'(A)\leq 2.$ Then $A$ contains a circuit $C$ of $M^X$ such that $|C| \leq 3.$ Since $\Gamma$ is independent in $M^X,$ $ C$ is not a subset of $\Gamma.$ Therefore $C$ contains $x$ and $ C - \{x\} \subset A -\{x \} \subset \Gamma.$   In the last row of the matrix $A^X$ which represents the matroid $M^X,$  the columns corresponding to the elements of $\Gamma$ have entries 1 and rest of the entries in that row are zero. As $C$ is a circuit, the sum of the columns of $A^X$ corresponding to the elements of $C$ is zero over GF(2). This implies that $C$ contains at least two elements of $\Gamma.$ Hence $ C =\{x, \gamma_1, \gamma_2 \}$ for some $\gamma_1, \gamma_2\in \Gamma.$ Let $x_1$ and $x_2$ be elements of the matroid $M$ corresponding to $\gamma_1$ and $\gamma_2,$ respectively. By Lemma \ref{2}(ii), $C_1 = \{x_1, x_2, \gamma_1, \gamma_2\}$ is a circuit in $M^X.$ Since $M^X$ is a binary matroid, the symmetric difference $ C \Delta C_1 =\{x, x_1, x_2 \}$ of the circuits $C$ and $C_1$ contains a circuit, say $C_2,$  of $M^X.$ Hence $C_2$ is a circuit in $M^X \setminus \Gamma = M$  such that $|C_2| \leq 3 = 4-1 = k-1,$  a contradiction by Lemma \ref{3}. Hence $r'(A)\geq  3.$ Since $|S_{1}|\leq k-3,$  by Corollary \ref{5}, $ r(S_2 ) = r(S - S_1) = r(M).$ Therefore, by Lemma \ref{1}(iii),  $r'(B)\geq r'(S_2) + 1 = r(S_2) + 1 = r(M)+1 = r'(M^X).$  Therefore $ r'(B) = r'(M^X).$  Hence  $  r'(A)+r'(B)-r'(M^{X}) = r'(A) \geq 3= k-1,$  a contradiction to (1).

Suppose $|S_2| \leq k-3.$ Then, as in the above paragraph, we see that $r'(B) \geq 3=k-1$ and $ r'(A) = r'(M^X)$ and  so $  r'(A)+r'(B)-r'(M^{X}) = r'(B) \geq k-1,$ a contradiction to (1).

Thus we get contradictions in Cases (i), (ii) and (iii). Therefore  $M^{X}$ is $k$-connected.

Conversely, suppose $M^{X}$ is $k$-connected.  The last row of the matrix $A^X,$ which represents $M^X,$ has 1's in the columns corresponding to the set $\Gamma$ and zero elsewhere.  Hence $\Gamma$ contains a cocircuit of $M^X.$  By Lemma \ref{3},  $|\Gamma| \geq k$ and so  $|X| = |\Gamma| \geq k.$ By Lemma \ref{2}(ii), $M^X$ contains a 4-circuit. Therefore, by Lemma \ref{3}, $k\leq 4.$  This completes the proof.
\end{proof}
We now give a necessary and sufficient condition to get a connected  matroid $M^X$ from the disconnected matorid $M.$  If  $X$ is  disjoint from a component $D$ of $M,$ then  it follows from Lemma \ref{2}  that $D$ is a component of $M^X$ also.  Therefore to get a  connected matroid $M^X$ from the disconnected matroid $M,$  it is necessary that $X$ intersects  every component of $M.$  In  the following theorem,  we prove that this obvious necessary condition is also suffcient.

\begin{theorem}
Let $M$ be a disconnected binary matroid and let $X$ be an independent set in $M.$  Then $M^X$ is connected if and only if  every component of $M$ intersects $X.$  
\end{theorem}
\begin{proof} Let $M_1, M_2,\dots, M_r$ be the components of $M.$  Suppose each $M_i$ intersects $X$.
Let $S$ be the ground set of $M.$ Then the ground set of $M^X$ is $S\cup \Gamma,$ where $S\cap \Gamma = \phi.$ Since  each  $M_i$ is connected in $M$ and  $ M^X\setminus \Gamma = M,$ each $M_i$ is connected in $M^X$  too. Therefore each $M_i$ is contained in a component of $M^X.$  We show that all $M_i$ are contained in a single component of $M^X.$  Since $M$ is disconnected, it has at least two components and so $r\geq 2.$ Let $D$ be a component of $M^X$ containing $M_1$ and let $ j\in  \{2, 3, \dots, r\}.$ Suppose $X$ contains an element $ x_1$ of $M_1$ and an element $ x_j$ of $M_j.$ Suppose $\gamma_1$ and $ \gamma_j$ are elements  of $\Gamma$ corresponding to $x_1$ and $x_j,$ respectively. Then, by Lemma 2.2(ii), $ C = \{x_1, x_j, \gamma_1, \gamma_j\}$ is  a 4-circuit in $M^X.$ As $C$ contains an element of the component $D$ of $M^X,$ $C$ is contained in $D.$ Therefore $D$ contains the element $x_j$ of $M_j.$ Consequently, $M_j$ is contained in $D.$ Thus all components of $M$ are contained in $D.$ Therefore $S\subset D.$  Let $\gamma$ be an arbitrary member of $\Gamma$ and let $x $ be the member of $X$ corresponding to $\gamma.$ Then, by Lemma 2.2(ii), $\gamma$ and $x$ belong to a  4-circuit, say $Z,$ of $M^X.$ As $x \in Z \cap D,$ $ Z \subset D$ and so $ \gamma \in D.$ Therefore $ \Gamma \subset D.$  Consequently, $D$ is the only component of $M^X.$ Hence $M^X$ is connected. 

The converse readily follows from the discussion prior to the statement of the theorem. 
\end{proof}
\vskip.2cm
\noindent
{\bf Example 2.8.}
 We illustrate Theorem \ref{6} by using the Fano matroid $F_7.$  The ground set of  $F_7$ is $\{1, 2, 3, 4, 5, 6, 7\}$ and the standard matrix representation of $F_7$ over $GF(2)$ is as follows:
 
$A = $\bordermatrix{ ~&1 &2&3&4&5&6&7\cr
                  ~&1& 0&0&0&1&1&1 \cr
                  ~&0&1& 0&1&0&1&1\cr
                  ~&0&0&1&1&1&0&1
                  }.  
		\vskip.2cm\noindent
Let $X  = \{1, 2\}$ and $ Y = \{1, 2, 3 \}.$ Then  $X$ and $Y$ are independent in $F_7.$ Further, \\

$A^X= $\bordermatrix{ ~&1 &2&3&4&5&6&7&\gamma_1&\gamma_2\cr
                  ~&1& 0&0&0&1&1&1&1&0 \cr
                  ~&0&1& 0&1&0&1&1&0&1\cr
                  ~&0&0&1&1&1&0&1&0&0 \cr
									~&0&0&0&0&0&0&0&1&1
                  } and   $A^Y= $\bordermatrix{ ~&1 &2&3&4&5&6&7&\gamma_1&\gamma_2&\gamma_3\cr
                  ~&1& 0&0&0&1&1&1&1&0&0 \cr
                  ~&0&1& 0&1&0&1&1&0&1&0\cr
                  ~&0&0&1&1&1&0&1&0&0&1 \cr
									~&0&0&0&0&0&0&0&1&1&1
                  }.
								\vskip.2cm\noindent
						Let $F_7^X$  and $F_7^Y$ be  the vector matroids of $A^X$ and $A^Y,$ respectively. It is well known that $F_7$ is 3-connected. One can check that $F_7^Y$ is 3-connected while $F_7^X$ is 2-connected but not 3-connected. 
\vskip.2cm

\noindent

\end{document}